\documentclass[reqno,12pt,a4letter]{amsart}
\usepackage{amsmath, amsxtra, amssymb, latexsym, amscd, amsthm,enumerate}
\usepackage{graphicx, color}
\usepackage[active]{srcltx}
\usepackage[utf8]{inputenc}
\usepackage[mathscr]{euscript}
\usepackage{mathrsfs,cite}
\usepackage[english]{babel}
\usepackage{color}
\usepackage[active]{srcltx}

\def\Limsup{\mathop{{\rm Lim}\,{\rm sup}}}

\def\dom{\mbox{\rm dom}\,}

\def\cl{\mbox{\rm cl}\,}

\def\R{\mathbb{R}}

\newtheorem {theorem}{Theorem}[section]

\newtheorem {proposition}{Proposition}[section]
\newtheorem {lemma}{Lemma}[section]
\newtheorem {example}{Example}[section]
\newtheorem {definition}{Definition}[section]
\newtheorem {remark}{Remark}[section]
  
\title{Optimality conditions at infinity for nonsmooth minimax programming}

\author{NGUYEN VAN TUYEN$^{1}$}
\address{$^1$Department of Mathematics, Hanoi Pedagogical University 2, Xuan Hoa, Phuc Yen, Vinh Phuc, Vietnam}
\email{nguyenvantuyen83@hpu2.edu.vn; tuyensp2@yahoo.com}
\author{KWAN DEOK BAE$^2$}
\address{$^2$Department of Applied Mathematics, Pukyong National University, Busan 48513, Republic of Korea}
\email{bkduck106@naver.com}

\author{DO SANG KIM$^3$}
\address{$^3$Department of Applied Mathematics, Pukyong National University, Busan 48513, Republic of Korea}
\email{dskim@pknu.ac.kr}

\date{\today}



\begin{document}

\maketitle

\begin{abstract}
This paper is devoted to study of optimality conditions at infinity in nonsmooth minimax programming problems and applications. By means of the limiting subdifferential and normal cone at infinity, we dirive necessary and sufficient optimality conditions of Karush--Kuhn--Tucker type for nonsmooth minimax programming problems with constraint. The obtained results are applied to a nonsmooth vector optimization problem.
\end{abstract}

\section{Introduction} \label{introduction}

In this paper, we are interested in  optimality conditions at infinity  for the following constrained minimax optimization problem
\begin{equation*} 
	\min_{x\in\mathcal{F}} \max_{1\leq i \leq m} f_i(x),
\end{equation*}
where  the feasible set $\mathcal{F}$ is defined by
\begin{equation*} 
	\mathcal{F}:=\{x\in \Omega\;|\; g_j(x)\leq 0, j=1, \ldots, p\},
\end{equation*}
$\Omega$ is a nonempty and closed subset of $\R^n$,   $f_i$ and $g_j$ are  lower semicontinuous functions. 

Minmax optimization problems have many applications in various areas such as Chebyshev’s theory of best approximation \cite{Demyanov},  game theory \cite{Akian,Chinchuluun,Ding,Emamizadeh,Tuy,Zhang}, vector optimization \cite{An-Ngoan-Tuyen-24,An-Ngoan-Tuyen-241,Chuong-Kim,Craven,Craven-2,Hong,Huyen-Kim-Yen-24}, ... and recently in  machine learning \cite{Jiang-Chen,Jiang-Chen-2,Qian-et al}.    

The study of optimality conditions is one of the most attractive topics in  optimization  and plays a crucial role  in both the theory and practice of   optimization problems. The classical Fermat's rule (or stationary principle) says that the derivative of a differentiable function vanishes at a point where the function attains its minimum.  The nonsmooth version of Fermat's rule   states that if a lower semicontinuous function $f\colon\R^n\to\overline{\mathbb{R}}:=\R\cup\{+\infty\}$ reaches its infimum at some point $\bar x\in\dom f$, then $0\in \partial^*  f(\bar x)$, where $\partial^*$ is a subdifferential of classical type; see, for example, \cite{Mordukhovich2006,Mordukhovich2018,Penot,Rockafellar1998}.

In contrast with the above case there are many others to be found in which the functions are bounded from below but  do not have a global minimum, for example, the function $f(x)=e^x$ for all $x\in\R$. Hence,  Fermat's rule cannot be applied for such cases.  However, we can see that if a bounded from below function $f\colon\R^n\to\overline{\mathbb{R}}$ does not attain its infimum, then  there exists a sequence $x_k$ tending to infinity such that $f(x_k)$ tends to its infimum. A natural question arises: {\em What kind of necessary optimality conditions can we  derive in these cases?}

Very recently, Nguyen and Pham \cite{Tung-Son-22} proposed the concepts of {\em Clarke’s} tangent and normal cones, subgradients, Lipschitz continuity, ..., {\em at infinity} and then they derived necessary optimality conditions for optimization problems in scenarios where objective functions are bounded from below but do not attain its minimum. Shortly thereafter, Kim, Nguyen and Pham \cite{Kim-Tung-Son-23}  introduced the notions of {\em limiting}  normal cones {\em at infinity} to unbounded sets as well as {\em limiting} and {\em singular} subdifferentials {\em at infinity} for extended real-valued functions. With the  aid  of these notions, the authors have succeeded in deriving necessary optimality conditions {\em at infinity} and studying weak sharp minima properties {\em at infinity} as well as stability results for optimization problems. It should be noted here that the concepts of  Clarke and   limiting subdifferentials at infinity are different and when a function is {\em Lipschitz at infinity}, then its limiting subdifferential    is strictly smaller than the Clarke subdifferential  at infinity; see \cite[Proposition 5.6]{Kim-Tung-Son-23}.  

To the best of our knowledge, so far there have been no papers investigating  optimality conditions at infinity for minimax optimization problems. This motivates us to study optimality conditions for such problems by using  the limiting subdifferential and normal cone   at infinity. 

The rest of the paper is organized as follows.  Section \ref{Preliminaries} contains some basic  definitions from variational analysis and several auxiliary results. In particular, we present necessary and sufficient at infinity for optimality  for a general optimization problem as well as derive a result for computing subdifferentials at infinity of maximum functions. In Section \ref{Section 3}, we present necessary  optimality conditions of Karush--Kuhn--Tucker type at infinity for minimax optimization problems. A sufficient condition at infinity for the  nonemptiness and compactness of the solution set of such problems is also given in this section. The application of the obtained results to study optimality conditions at infinity for vector  optimization problems is established in Section \ref{Section 4}. In the last section, we give several concluding remarks.

\section{Preliminaries}\label{Preliminaries}
\markboth{\centerline{\it Optimality conditions at infinity for nonsmooth minimax programming}}{\centerline{\it N.V.~Tuyen and D.S.~Kim}} \setcounter{equation}{0}

\subsection{Normal cones and subdifferentials}
Let us recall some notions related to generalized differentiation from~\cite{Mordukhovich2006,Mordukhovich2018,Rockafellar1998}. Throughout the paper, we deal with the Euclidean space $\R^n$ equipped with the usual scalar product $\langle \cdot, \cdot\rangle$ and the corresponding norm $\|\cdot\|$. The closed ball centered at $x$ with radius $r$ is denoted by $\mathbb{B}_r(x)$; when $x$ is the origin of $\R^n$ we write $\mathbb{B}_r$ instead of $\mathbb{B}_r(x)$.   The topological interior, the convex hull, and the positive hull of a set   $D\subset \R^n$ are  denoted, respectively, by $\mathrm{int}\, D$, $\mathrm{co}\, D$, and $\mathrm{pos}\, D$.       

Let  $F : \R^n \rightrightarrows \R^m$ be a set-valued mapping. The \textit{domain} and the \textit{graph}  of $F$ are given, respectively, by
$${\rm dom}\,F=\{x\in \R^n \mid F(x)\not= \emptyset\} $$
and
$$ {\rm gph}\,F=\{(x,y)\in \R^n \times \R^m \mid y \in F(x)\}.$$
We say that $F$ is \textit{proper} if $\dom F \not= \emptyset.$ The limiting construction
\begin{align*}
	\Limsup\limits_{x\rightarrow \bar x} F(x):=\bigg\{ y\in \mathbb{R}^m \mid \exists x_k \rightarrow \bar x, y_k \rightarrow y \ \mbox{with}\ y_k\in F(x_k), \forall k=1,2,....\bigg\}
\end{align*}
is known as the \textit{Painlev\'e--Kuratowski outer/upper limit} of $F$ at $\bar x$. 

\begin{definition}[{see \cite{Mordukhovich2006,Mordukhovich2018}}]{\rm 
		Let $\Omega$ be a nonempty subset of $\mathbb{R}^n$ and $\bar x \in \Omega$. 
		\begin{enumerate}[(i)]
			\item The \textit{regular/Fr\'echet normal cone} to $\Omega$ at $\bar x$ is defined by
			\begin{align*}
				\widehat N(\bar x, \Omega)=\left\{ v\in \mathbb{R}^n\mid \limsup\limits_{x \xrightarrow{\Omega}\bar x} \dfrac{\langle v, x-\bar x \rangle}{\|x-\bar x\|} \leq 0 \right\},
			\end{align*}
			where $x \xrightarrow{\Omega} \bar x$ means that $x \rightarrow \bar x$ and $ x\in \Omega$.
			\item The \textit{limiting/Mordukhovich  normal cone} to $\Omega$ at $\bar x$ is given by
			\begin{align*}
				N(\bar x, \Omega)=\Limsup\limits_{ x \xrightarrow{\Omega} \bar x} \widehat{N}(x, \Omega).
			\end{align*}
			We put $\widehat N(\bar x,\Omega)=N(\bar x,\Omega) =\emptyset$ if $\bar x \not\in \Omega$.
	\end{enumerate} }
\end{definition}	
Clearly, one always has
\begin{align*} 
	\widehat N(x,\Omega) \subset N(x,\Omega),\ \ \   \forall  x \in \Omega.
\end{align*}
Consider a function $f: \mathbb{R}^n\rightarrow \overline{\mathbb{R}}:=\mathbb{R}\cup\{+\infty\}$ with the \textit{effective domain} $$\mbox{dom}\, f:=\{x \in \mathbb{R}^n \mid f(x) < +\infty\}$$ and the \textit{epigraph} $$ {\rm{epi}}\, f:=\{ (x, \alpha) \in \mathbb{R}^n \times \mathbb{R} \mid \alpha \ge f (x)\}.$$  
We call $f$ a {\em proper} function if its $\mathrm{dom} f$ is a nonempty set.

\begin{definition}[{see \cite{Mordukhovich2006,Mordukhovich2018}}]\label{def21} {\rm Consider a function  $f \colon \mathbb{R}^n \to \overline{\mathbb{R}}$ and a point $\bar{x} \in {\rm dom}f$.  
		\begin{enumerate}[{\rm (i)}]
			\item The {\em regular/Fr\'echet subdifferential} of $f$ at $\bar{x}$ is 
			$$
			\widehat{\partial}f(\bar{x}):=\{ v \in \mathbb{R}^n \mid (v,-1)\in \widehat{N}_{\mathrm{epi} f}(\bar{x},f(\bar{x}))  \}.  
			$$
			\item The {\em limiting/Mordukhovich subdifferential} of $f$ at $\bar{x}$ is 
			$$
			\partial f(\bar{x}):=\{ v \in \mathbb{R}^n \mid (v,-1)\in {N}_{\mathrm{epi} f}(\bar{x},f(\bar{x}))  \}, 
			$$
			and the {\em limiting/Mordukhovich singular subdifferential} of $f$ at $\bar{x}$ is 
			$$\partial^{\infty} f(\bar{x}):=\{ v \in \mathbb{R}^n \mid \exists v_k \in 	\widehat{\partial} f(\bar{x}), \lambda_k \downarrow 0, \lambda_k v_k \to v \}.	$$
	\end{enumerate}}
\end{definition}
It is well-known that
$$
\partial f(\bar{x})=\Limsup_{x \xrightarrow{f} \bar{x}}\widehat{\partial}f(x) \supseteq \widehat{\partial}f(x),
$$
where $x \xrightarrow{f} \bar{x}$ means that $x \to \bar{x}$ and $f(x) \to f(\bar{x})$. In
particular, if $f$ is a convex function, then the subdifferentials $\widehat{\partial} f(\bar{x})$ and $\partial f(\bar{x})$
coincide with the subdifferential in the sense of convex analysis. 

For the singular subdifferential, we have
\begin{align*}
	\partial^{\infty} f(\bar{x})\subseteq \{ v \in \mathbb{R}^n \mid (v,0)\in N_{\mathrm{epi} f}(\bar{x},f(\bar{x}))  \},
\end{align*}
see, for example, \cite{Rockafellar1998}. This relationship holds with equality whenever $f$ is lower semicontinuous (l.s.c.) at $\bar{x}$.

Let $\Omega \subset \mathbb{R}^n$, we define the {\em indicator function} $\delta_{\Omega} \colon \mathbb{R}^n \to \overline{\mathbb{R}}$ by
$$\delta_{\Omega}(x) :=
\begin{cases}
	0 & \textrm{ if } x \in \Omega, \\
	+\infty & \textrm{ otherwise.}
\end{cases}$$
It holds that $\partial \delta_{\Omega} (x) =\partial^{\infty} \delta_{\Omega} (x)= N_{\Omega}(x)$ for any $x \in \Omega.$

Some calculus rules for the limiting/Mordukhovich and singular subdifferentials used later are collected in the following lemmas (see \cite{Mordukhovich2006,Mordukhovich2018,Rockafellar1998}).

\begin{lemma}[{Fermat rule, see \cite[Proposition 1.114]{Mordukhovich2006}}] \label{lema22}
	If a proper function $f \colon \mathbb{R}^n \to \overline{\mathbb{R}}$ has a local minimum at $\bar{x},$ then $0 \in \widehat{\partial}f(\bar{x})\subset \partial f(\bar{x}).$	
\end{lemma}

\begin{lemma}[{see \cite[Theorem 4.10]{Mordukhovich2018}}]\label{lema23} 
	Let $f_i \colon \mathbb{R}^n \to \overline{\mathbb{R}}$, $i=1,\dots,m$ with $m \geq 2$, be Lipschitz around $\bar{x}$. Then the subdifferential of maximum function holds 
	\begin{equation*}
		\partial (\max f_i) (\bar{x}) \subseteq \left\{    \sum_{i \in I(\bar{x})} \lambda_i \partial f_i(\bar{x}) \mid \lambda_i \ge 0, \sum_{i \in I(\bar{x})} \lambda_i=1  \right\}, 
	\end{equation*}	
	where $I(\bar{x}):=\{ i \in \{1,\ldots,m\}\mid f_i(\bar{x})= (\max f_i) (\bar{x})\}$. 
\end{lemma}
We now recall the Ekeland variational principle (see \cite{Ekeland-74}).
\begin{lemma}[Ekeland variational principle]\label{lema26} 
	Let $f \colon \mathbb{R}^n \to \overline{\mathbb{R}}$ be a proper l.s.c. function and bounded from below. Let $\epsilon >0$ and $x_0 \in \mathbb{R}^n$ be given such that 
	\begin{eqnarray*}
		f (x_0) &\le& \inf_{x \in \mathbb{R}^n}f (x) + \epsilon. 
	\end{eqnarray*}
	Then, for any $\lambda >0$ there is a point $x_1 \in \mathbb{R}^n$ satisfying the following conditions
	\begin{enumerate}[{\rm (i)}]
		\item $f (x_1) \le f(x_0),$
		\item $\|x_1-x_0\| \le \lambda,$ and
		\item $f (x_1) \le  f(x)+\dfrac{\epsilon}{\lambda}\|x-x_1\|$ for all $x \in \mathbb{R}^n.$
	\end{enumerate}
\end{lemma}

\subsection{Normal cones and subdifferentials at infinity}
Let $\Omega$ be a {\em locally closed} subset of $\mathbb{R}^n$, i.e., for any $x\in \Omega$ there is a neighborhood $U$ of $x$ for
which $\Omega\cap U$ is closed.  Let $I$ be a nonempty subset of $\{1, \ldots, n\}.$ Consider the projection $\pi \colon \mathbb{R}^n \to \mathbb{R}^{\#I}, x := (x_1, \ldots, x_n) \mapsto (x_i)_{i \in I}$ and assume that the set $\pi(\Omega)$ is unbounded. In what follows, the notation $ \pi(x) \xrightarrow{\Omega} \infty $ means that $x \in \Omega$ and $\|\pi(x)\| \to \infty$. 

\begin{definition}[{see \cite{Kim-Tung-Son-23}}]\label{def31} {\rm 
		The {\em normal cone to the set $\Omega$ at infinity (with respect to the index set $I$)} is defined by
		\begin{eqnarray*}
			N_{\Omega}(\infty_I) &:=& \Limsup_{\pi(x) \xrightarrow{\Omega} \infty} \widehat{N}_{\Omega}(x).
		\end{eqnarray*} 
		When $I = \{1, \ldots, n\},$ we write $N_{\Omega}(\infty)$ instead of $N_{\Omega}(\infty_I).$
}\end{definition}
The following result gives  the intersection rule for
normals at infinity.
\begin{proposition}[{see \cite[Proposition 3.7]{Kim-Tung-Son-23}}]\label{intersection-normal-cones}
	Let $\Omega_1, \Omega_2$ be locally closed subsets of $\mathbb{R}^n$ satisfying the normal qualification condition at infinity
	\begin{eqnarray*}
		N_{\Omega_1}(\infty_I) \cap \big (-N_{\Omega_2}(\infty_I) \big) &=& \{0\}.
	\end{eqnarray*}
	Then we have the inclusion
	\begin{eqnarray*}
		N_{\Omega_1 \cap \Omega_2}(\infty_I)  &\subset& N_{\Omega_1}(\infty_I) + N_{\Omega_2}(\infty_I).
	\end{eqnarray*}
\end{proposition}

Now let $f \colon \mathbb{R}^n \to \overline{\mathbb{R}}$ be a lower semicontinuous function and let $I := \{1, \ldots, n\} \subset \{1, \ldots, n, n + 1\}.$ To avoid triviality, we will assume that $f$ is {\em proper at infinity} in the sense that the set $\mathrm{dom} f$ is unbounded.

\begin{definition}[{see \cite{Kim-Tung-Son-23}}]\label{def41} {\rm 
		The {\em limiting/Mordukhovich and the singular subdifferentials} of $f$ at infinity are defined, respectively, by 
		\begin{eqnarray*}
			\partial f(\infty) &:=& \{u \in \mathbb{R}^n \ | \ (u, -1) \in \mathcal{N} \},\\
			\partial^{\infty} f(\infty) &:=& \{u \in \mathbb{R}^n \ | \ (u, 0) \in \mathcal{N}\},
		\end{eqnarray*}
		where $\mathcal{N} := \displaystyle \Limsup_{x \to \infty} {N}_{\textrm{epi} f}(x,f(x)).$
}\end{definition}
The following result gives limiting representations of limiting and singular subgradients at infinity.
\begin{proposition}[{see \cite[Proposition 4.4]{Kim-Tung-Son-23}}]\label{pro42} 
	The following relationships hold
	\begin{eqnarray*}
		\partial f(\infty) &=& \Limsup_{x \to \infty} \partial f(x), \label{PT421} \\
		\partial^{\infty} f(\infty)  &=& \Limsup_{x \to \infty, r \downarrow 0} r \partial f(x), \label{PT422}  \\
		\partial^{\infty} f(\infty)  &\supseteq& \Limsup_{x \to \infty} \partial^{\infty} f(x) \label{PT423} .
	\end{eqnarray*}
\end{proposition}
\begin{remark}
	{\rm Let $\Omega$ be an unbounded and closed  subset of $\R^n$. Then by Proposition \ref{pro42} and \cite[Proposition 1.19]{Mordukhovich2018}, we have
		\begin{equation*} 
			\partial\delta_\Omega(\infty)=\partial^\infty\delta_\Omega(\infty)=N_\Omega(\infty).
		\end{equation*}  
		
	}
\end{remark}

We now recall the notion of the Lipschitzness at infinity for l.s.c. functions (see \cite{Kim-Tung-Son-23,Tung-Son-22}). 
\begin{definition}\label{def51}{\rm
		Let  $f \colon \mathbb{R}^n \to \mathbb{R}$ be a l.s.c function.  We say that $f$ is {\em Lipschitz at infinity} if there exist constants $L > 0$ and $R > 0$ such that
		\begin{eqnarray*}
			|f(x) - f(x')| &\le& L \|x - x'\| \ \  \textrm{ for all } \ \ x, x' \in \mathbb{R}^n \setminus \mathbb{B}_R.
		\end{eqnarray*}
}\end{definition}
The following result gives a  characterization for the  Lipschitzness at infinity of l.s.c. functions. 
\begin{proposition}[{see \cite[Proposition 5.2]{Kim-Tung-Son-23}}]\label{pro52}   
	Let $f \colon \mathbb{R}^n \to \mathbb{R}$ be a l.s.c. function. Then $f$ is Lipschitz at infinity if and only if $\partial^{\infty}f(\infty)=\{0\}.$ In this case, $\partial f(\infty)$ is nonempty compact. 
\end{proposition}
The next result presents sum rules for both basic
and singular  subdifferentials at infinity.
\begin{proposition}[{cf. \cite[Proposition 4.9]{Kim-Tung-Son-23}}] \label{pro54}
	Let $f_1, \ldots, f_m \colon \mathbb{R}^n\to\overline{\mathbb{R}}$ be l.s.c. functions such that the following subdifferential qualification condition at infinity
	\begin{equation}\label{qualification-condition}
		\left[u_1+\ldots+u_m=0, u_i\in\partial^{\infty} f_i(\infty)\right] \Rightarrow u_i=0, i=1, \ldots, m,
	\end{equation}
	is valid. Then one has the  inclusions
	\begin{eqnarray*}
		\partial (f_1 +\ldots+ f_m)(\infty) &\subset& \partial f_1(\infty) +\ldots+ \partial  f_m(\infty), \\
		\partial^{\infty} (f_1 +\ldots+ f_m)(\infty) &\subset& \partial^{\infty} f_1(\infty) + \ldots+ \partial^{\infty} f_m(\infty).   
	\end{eqnarray*}
\end{proposition}
\begin{remark}
	{\rm Clearly, by Proposition \ref{pro52}, the subdifferential qualification condition \eqref{qualification-condition} holds at infinity if functions $f_1, \ldots, f_m$, except, possibly, one are Lipschitz at infinity. 
		
	}
\end{remark}

We now present a necessary optimality condition at infinity for a general optimization problem of the form
\begin{equation}\label{problem-0}
	\min_{x\in\Omega} f(x),\tag{P$_0$}
\end{equation}
where $f \colon \mathbb{R}^n \to \overline{\mathbb{R}}$ is a l.s.c. function and  $\Omega$ is a nonempty and closed subset of $\R^n$.

\begin{theorem}[{Fermat rule, cf. \cite[Theorem 6.1]{Kim-Tung-Son-23}}] \label{theo56} 
	Assume that the following conditions hold:
	\begin{enumerate}[\rm (i)]
		\item $\mathrm{dom} f \cap \Omega$ is unbounded.
		\item $\partial^{\infty}f(\infty) \cap \big(-N_{\Omega}(\infty) \big) =\{0\}.$
		\item $f$ is bounded from below on $\Omega,$ i.e., $f_* := \inf_{x\in \Omega} f(x)$ is finite.
	\end{enumerate}
	If there exists a sequence $x_k\xrightarrow{\Omega} \infty$ such that $f(x_k)\to f_*$, then 
	\begin{eqnarray*}
		0 &\in& \partial f(\infty) + N_{\Omega}(\infty).
	\end{eqnarray*}
\end{theorem} 
\begin{proof} 
	We first consider the case that $\Omega=\mathbb{R}^n$. Then  $N_{\Omega}(\infty) = \{0\}.$ For each $k\in\mathbb{N}$, we have
	\begin{eqnarray*}
		f_*\leq f(x_k)\leq f_*+\left(f(x_k)-f_*+\frac{1}{k}\right). 
	\end{eqnarray*}
	Clearly, $\epsilon_k:=f(x_k)-f_*+\frac{1}{k}>0$  and $\epsilon_k\to 0$ as $k\to\infty$. Put $\lambda_k:=\sqrt{\epsilon_k}$, then by the Ekeland variational principle (Lemma~\ref{lema26}), there exists $x^\prime_k \in \mathbb{R}^n$ for $k > 0$ such that
	\begin{eqnarray*}
		\|x_k - x^\prime_k\| & \le & \lambda_k, \\
		f(x^\prime_k) &\le& f(x)+\lambda_k\|x-x^\prime_k\| \quad \textrm{for all } \quad x \in \mathbb{R}^n.
	\end{eqnarray*}
	The first inequality and the fact that  $x_k\to\infty$  imply that the sequence $x^\prime_k~\to~\infty$. While the second inequality says that $x^\prime_k$ is a global minimizer of the function $\varphi(\cdot):=f(\cdot)+\lambda_k\|\cdot-x^\prime_k\|$  on $\mathbb{R}^n$. 	By the Fermat rule (Lemma~\ref{lema22}), we obtain
	\begin{eqnarray*}
		0 &\in& \partial \left( f(\cdot)+ \lambda_k\|\cdot-x^\prime_k\|\right)(x^\prime_k).
	\end{eqnarray*}
	By the Lipschitz property of the function $\|\cdot-x^\prime_k\|$ and the sum rule (see \cite[Theorem~2.19]{Mordukhovich2018}), we have 
	\begin{align*}
		0 &\in \partial f(x^\prime_k)+ \lambda_k\partial (\|\cdot-x^\prime_k\|)(x^\prime_k)
		\\
		&=\partial f(x^\prime_k)+ \lambda_k\mathbb{B}
	\end{align*}
	due to $\partial (\|\cdot-x^\prime_k\|)(x^\prime_k)=\mathbb{B}$. Hence, 
	$$0 \in \partial f(x^\prime_k)+ \lambda_k\mathbb{B},$$
	and so there is ${u}_k \in \partial f(x^\prime_k)$ such that $\|{u}_k\| \leq \lambda_k$. Since $\lambda_k\to 0$, by letting $k \to \infty$ and applying Proposition~\ref{pro42}, we obtain  $ 0\in \partial f(\infty).$
	
	We now consider the case that $\Omega$ is an arbitrary subset of $\mathbb{R}^n$. 
	We have
	\begin{eqnarray*}
		f_*=\inf_{x \in \mathbb{R}^n} \left(f + \delta_{\Omega}\right)(x) & = & \inf_{x \in \Omega} f(x) \ > \ -\infty,
	\end{eqnarray*}
	where $\delta_{\Omega} \colon \mathbb{R}^n \to \overline{\mathbb{R}}$ stands for the indicator function of the set $\Omega.$ Clearly, $f(x_k) + \delta_{\Omega}(x_k)\to f_*$. Therefore, $0 \in \partial (f + \delta_{\Omega})(\infty)$ (by the argument employed in the first case). This, together with Proposition~\ref{pro54} and the assumption (ii), yields $0 \in \partial f (\infty) + N_{\Omega}(\infty).$
\end{proof}

The following theorem gives a sufficient condition at infinity for the nonemptiness and compactness of \eqref{problem-0}.
\begin{theorem}\label{abstract-sufficient-theorem} Assume that condition (ii) of Theorem \ref{theo56} is satisfied. If $f$ is bounded from below on $\Omega$ and 
	\begin{eqnarray}\label{sufficient-condition-1}
		0 &\notin& \partial f(\infty) + N_{\Omega}(\infty),
	\end{eqnarray}
	then the global solution set of \eqref{problem-0}, denoted by $\mathrm{sol}\eqref{problem-0}$, is nonempty and compact.
\end{theorem}
\begin{proof}
	Since $f$ is bounded from below on $\Omega,$  $f_* := \inf_{x\in \Omega} f(x)$ is finite. By the definition of infimum, there exists a sequence $x_k\in \Omega$ such that $f(x_k)\to f_*$ as $k\to\infty$. If the condition \eqref{sufficient-condition-1} holds, then by Theorem \ref{theo56} the sequence $x_k$ is bounded. By passing to a subsequence if necessary we may assume that $x_k$ converges to some $\bar x$. Since  $\Omega$ is closed,  $\bar x\in \Omega$. By the lower semicontinuity of $f$, we have
	\begin{equation*}
		f(\bar x)\leq\liminf_{x\to\bar x} f(x) \leq \lim_{k\to\infty} f(x_k)=f_*.
	\end{equation*}
	This implies that $f(\bar x)=f_*$, or, equivalently, $\bar x$ is a global minimum of $f$ on $\Omega$.  This means that  $\mathrm{sol}\eqref{problem-0}$ is nonempty. Furthermore, we see that 
	$$\mathrm{sol}\eqref{problem-0}=\{x\in\Omega\;|\; f(x)\leq f_*\}=\{x\in\Omega\;|\; f(x)= f_*\}.$$ 
	This set is closed due to the lower semicontinuity of  $f$ and the closedness of $\Omega$. By Theorem \ref{theo56} and \eqref{sufficient-condition-1}, $\mathrm{sol}\eqref{problem-0}$ is bounded and so it is compact.  The proof is complete. 
\end{proof}
\subsection{Subdifferential of maximum functions at infinity} 
In this subsection, we proceed with evaluating limiting and singular subdifferentials at infinity of  maxima of finitely many functions.
\begin{theorem}\label{max-rule-theorem}
	Let $f_1, \ldots, f_m\colon \mathbb{R}^n\to \R$ be     Lipschitz at infinity functions. Then we have
	\begin{align}
		\partial (\max f_i) (\infty) &\subseteq \mathrm{co}\,\left\{\partial f_1(\infty), \ldots,  \partial f_m(\infty)\right\}, \label{max-rule-infinity}
		\\
		\partial^\infty (\max f_i) (\infty) &=\{0\}. \notag
	\end{align} 
\end{theorem}
\begin{proof} Since $f_i$, $i=1, \ldots, m$, are Lipschitz at infinity, it is easy to check that $(\max f_i)$ is also Lipschitz at infinity. So by Proposition \ref{pro52}, one has $\partial^\infty (\max f_i) (\infty) =\{0\}$. We now take any $z\in\partial (\max f_i) (\infty)$, then there exist sequences $x_k\to \infty$ and $z_k\in  \partial (\max f_i) (x_k)$ such that $z_k\to z$ as $k\to\infty$.  Since $f_i$, $i=1, \ldots, m$, are Lipschitz at infinity, we can assume that these functions are locally Lipschitz around $x_k$ with the same modulus $L>0$ for all $k$ large enough. Hence, $\|u\|\leq L$ for all $u\in \partial f_1(x_k)$,  $i=1, \ldots, m$ and $k$ large enough.  By Lemma \ref{lema23}, we have
	\begin{equation*}
		\partial (\max f_i) (x_k) \subseteq \mathrm{co}\,\left\{\partial f_1(x_k), \ldots,  \partial f_m(x_k)\right\}.
	\end{equation*}
	Hence, there exist $\lambda^i_k\geq 0$, $i=1, \ldots, m$, with $\sum_{i=1}^{m}\lambda^i_k=1$ and $u^i_k\in \partial f_i(x_k)$, $i=1, \ldots, m$ such that $z_k=\sum_{i=1}^{m}\lambda^i_ku^i_k$. By passing to a subsequence if necessary we may assume that $\lambda^i_k\to \lambda_i$ and $u^i_k\to u_i$, $i=1, \ldots, m$, as $k\to\infty$. Hence, $u_i\in\partial^\infty f_i(\infty)$,    $\sum_{i=1}^{m}\lambda_i=1$ and $z=\sum_{i=1}^{m}\lambda_i u_i$. This means that $z\in \mathrm{co}\,\left\{\partial f_1(\infty), \ldots,  \partial f_m(\infty)\right\}$ and thus \eqref{max-rule-infinity} holds. The proof is complete. 
\end{proof}
\section{Optimality conditions at infinity}\label{Section 3}

Let $\Omega$ be a nonempty and closed  subset of $\R^n$, $I:=\{1, \ldots, m\}$ and $J:=\{1, \ldots, p\}$. Let us consider a  minimax programming problem of the following form
\begin{equation}\label{problem}
	\min_{x\in\mathcal{F}} \max_{i\in I} f_i(x),\tag{P}
\end{equation}
where  the feasible set $\mathcal{F}$ is defined by
\begin{equation}\label{feasible-set}
	\mathcal{F}:=\{x\in \Omega\;|\; g_j(x)\leq 0, j\in J\}
\end{equation}
and   $f_i, i\in I$, and $g_j, j\in J$, are  l.s.c functions. 
Denote $f:=(f_1, \ldots, f_m)$, $g:=(g_1, \ldots, g_p)$ and $\varphi:=\max_{i\in I} f_i$. 

{\em Hereafter}, we assume that $\mathcal{F}$ is nonempty and {\em unbounded set}, $f$ and $g$ are {\em Lipschitz at infinity}, and the maxima function $\varphi$ is {\em bounded from below} on $\mathcal{F}$. 

\begin{definition}
	{\rm We say that {\em the constraint qualification condition} (CQ) holds at infinity if there do not exist $\beta_j\geq 0$, $j\in J$, such that $\sum_{j\in J}\beta_j>0$ and 
		\begin{equation}\label{constraint-qualification}
			0\in \sum_{j\in J}\beta_j\partial g_j(\infty)+N_\Omega(\infty).
		\end{equation}
	}
\end{definition}

The following result gives a necessary optimality condition of Karush--Kuhn--Tucker (KKT) type at infinity to problem \eqref{problem}.
\begin{theorem}\label{KKT-necessary-theorem} Assume that the condition (CQ) holds at infinity. If there is a sequence $x_k\xrightarrow{\mathcal{F}}\infty$ such that $\varphi (x_k)\to\varphi_*:=\inf_{x\in \mathcal{F}}\varphi(x)$, then there exist $\alpha:=(\alpha_1, \ldots, \alpha_m)\in\R^m_+\setminus\{0\}$ and $\beta:=(\beta_1, \ldots, \beta_p)\in\R^p_+$ such that  
	\begin{align}\label{equa-7}
		0\in \sum_{i\in I} \alpha_i\partial f_i(\infty)+\sum_{j\in J}\beta_j\partial g_j(\infty)+N_\Omega(\infty).
	\end{align}
\end{theorem}  
\begin{proof} Since $f_i$, $i\in I$, are Lipschitz at infinity, $\varphi$ is also Lipschitz at infinity. Hence, all assumptions of Theorem \ref{theo56} are satisfied. Then, one has
	\begin{equation}\label{equa-6}
		0\in \partial \varphi(\infty)+N_{\mathcal{F}}(\infty).
	\end{equation}
	By Theorem \ref{max-rule-theorem}, we have
	\begin{equation}\label{max-rule-subdifferential}
		\partial \varphi (\infty)  \subseteq \mathrm{co}\,\left\{\partial f_1(\infty), \ldots,  \partial f_m(\infty)\right\}. 
	\end{equation} 
	On the other hand, by letting 
	\begin{equation*}
		C:= \{x\in\R^n\;|\; g_j(x)\leq 0, j\in J\},
	\end{equation*}
	then one has $\mathcal{F}=C\cap \Omega$. 
	
	We first claim that
	\begin{equation}\label{equa-4}
		N_C(\infty)\subset \mathrm{pos}\,\{\partial g_1(\infty), \ldots, \partial g_p(\infty)\},
	\end{equation}
	where 
	$$\mathrm{pos}\,\{\partial g_1(\infty), \ldots, \partial g_p(\infty)\}:=\left\{\sum_{i=1}^p\beta_j v_j\,|\, \beta_j\geq 0, v_j\in \partial g_j(\infty), j=1, \ldots, p\right\}.$$
	Indeed, take any $u\in 	N_C(\infty)$. Clearly, $u=0\in \mathrm{pos}\,\{\partial g_1(\infty), \ldots, \partial g_p(\infty)\}$. So we consider the case that $u\neq 0$.  Then there exist sequence $x_k\xrightarrow{C}\infty$ and $u_k\in N_C(x_k)$ such that $u_k\to u$ as $k\to\infty$. Without loss of generality, we may assume that $u_k\neq 0$ for all $k\in\mathbb{N}$. Since $g_j$, $j\in J$, are Lipschitz at infinity and $x_k\xrightarrow{C}\infty$, we may assume that $g_j$, $j\in J$, are locally Lipschitz with the same modulus $L>0$ at $x_k$ for all $k$ large enough.  Then by \cite[Corollary 4.36]{Mordukhovich2006} one has
	\begin{equation*}
		N_C(x_k)\subset \mathrm{pos}\,\{\partial g_1(x_k), \ldots, \partial g_p(x_k)\}
	\end{equation*}
	for all $k$ large enough, i.e., there exist $\beta^j_k\geq 0$ and $v^j_k\in \partial g_j(x_k)$, $j\in J$, such that
	\begin{equation}\label{equa-2}
		u_k=\sum_{j\in J}\beta^j_kv^j_k.
	\end{equation}
	Since $u_k\neq 0$, one has $\beta_k:=\sum_{j\in J}\beta^j_k>0$. Moreover, by the locally Lipschitz property at $x_k$  of $g_j$, $j\in J$,  
	$\|v^j_k\|\leq L$ for all $j\in J$ and $k$ large enough.   Dividing two sides of \eqref{equa-2} by $\beta_k$ we obtain
	\begin{equation}\label{equa-3}
		\dfrac{u_k}{\beta_k}= \sum_{j\in J}\frac{\beta^j_k}{\beta_k}v^j_k.
	\end{equation}
	Since sequences $\frac{\beta^j_k}{\beta_k}$ and $v^j_k$ are bounded, by passing to a subsequence if necessary we can assume that $\frac{\beta^j_k}{\beta_k}\to \beta_j\geq 0$ and $v^j_k\to v_j\in\partial g_j(\infty)$.  
	
	If the sequence $\beta_k$ is unbounded, then $\frac{u_k}{\beta_k}\to 0$. Since $\sum_{j\in J}\frac{\beta^j_k}{\beta_k}=1$, one has $\sum_{j\in J}\beta_j=1$. Now by letting $k\to\infty$ in \eqref{equa-3} we obtain
	\begin{equation*}
		0=\sum_{j\in J}\beta_jv_j,
	\end{equation*}
	contrary to \eqref{constraint-qualification}. Thus the sequence $\beta_k$ is bounded. By passing to a subsequence if necessary we can assume that $\beta_k\to \beta\geq 0$. If $\beta=0$, then it follows from \eqref{equa-2} that $u=\lim_{k\to\infty}u_k=0$, a contradiction. So $\beta>0$. Letting $k\to\infty$ in \eqref{equa-3} we obtain
	\begin{equation*}
		u=\beta\sum_{j\in J}\beta_jv_j\in \mathrm{pos}\,\{\partial g_1(\infty), \ldots, \partial g_p(\infty)\},
	\end{equation*}
	as required. 
	
	We now show that $N_C(\infty)\cap(-N_\Omega(\infty))=\{0\}$ and hence by Proposition \ref{intersection-normal-cones} and \eqref{equa-4} we have
	\begin{align}
		N_\mathcal{F}(\infty)&\subset N_C(\infty)+N_\Omega(\infty)\notag
		\\
		&\subset \mathrm{pos}\,\{\partial g_1(\infty), \ldots, \partial g_p(\infty)\}+N_\Omega(\infty).\label{equa-5}
	\end{align} 
	Indeed, let $u\in N_C(\infty)\cap(-N_\Omega(\infty))$. Then by \eqref{equa-4}, there exist $\beta_j\geq 0$ and $v_j\in \partial g_j(\infty)$, $j\in J$, such that $u=\sum_{j\in J}\beta_j v_j$. Hence,
	\begin{equation*}
		0= \sum_{j\in J}\beta_j v_j-u.
	\end{equation*}
	Since the condition (CQ) holds at infinity, we deduce that $\beta_j=0$, $j\in J$, or, equivalently,  $u=0$, as required.
	
	Combining  \eqref{equa-6}, \eqref{max-rule-subdifferential} and \eqref{equa-5} we obtain
	\begin{equation*}
		0\in \mathrm{co}\,\left\{\partial f_1(\infty), \ldots,  \partial f_m(\infty)\right\}+\mathrm{pos}\,\{\partial g_1(\infty), \ldots, \partial g_p(\infty)\}+N_\Omega(\infty).
	\end{equation*}
	This means that there exist $\alpha:=(\alpha_1, \ldots, \alpha_m)\in\R^m_+\setminus\{0\}$ and $\beta:=(\beta_1, \ldots, \beta_p)\in\R^p_+$ satisfying \eqref{equa-7}. The proof is complete. 
\end{proof}
Let us see a simple example to illustrate the above theorem.
\begin{example}
	{\rm Let  $f_1, f_2, g\colon \R\to \R$ be functions defined respectively by $f_1(x)=\frac{1}{|x|+1}$, $f_2(x)=0$, $g(x)=x$ and let $\Omega=\R$. Is is easy to see that $f_1$, $f_2$ and $g$ are Lipschitz at infinity, $\partial f_1(\infty)=\partial f_2(\infty)=0$, $\partial g(\infty)=1$ and $N_\Omega(\infty)=\{0\}$. Hence, the condition (CQ) holds at infinity. Furthermore, 
		$$\varphi(x)=\max\{f_1(x), f_2(x)\}=\frac{1}{|x|+1}\ \ \ \forall x\in\R.$$
		Clearly, $\varphi_*=\inf_{x\in \mathcal{F}}\varphi(x)=0$ and $\varphi$ does not attain its infimum on $\mathcal{F}$. Then by Theorem \ref{KKT-necessary-theorem}, there exist $(\alpha_1, \alpha_2)\in\R^2_+\setminus\{0\}$ and $\beta\geq 0$ such that
		\begin{equation*}
			0\in \alpha_1\partial f_1(\infty)+\alpha_2\partial f_2(\infty)+\beta \partial g(\infty).
		\end{equation*}
	}
\end{example}
The following theorem gives   sufficient conditions for the nonempty and compactness of the global solution set of \eqref{problem}.   This result is a direct consequence of Theorems \ref{abstract-sufficient-theorem} and \ref{KKT-necessary-theorem}.
\begin{theorem}
	If the condition (CQ) holds at infinity and  
	\begin{equation*}
		0\notin \mathrm{co}\,\left\{\partial f_1(\infty), \ldots,  \partial f_m(\infty)\right\}+\mathrm{pos}\,\{\partial g_1(\infty), \ldots, \partial g_p(\infty)\}+N_\Omega(\infty),
	\end{equation*}
	then the global solution set of  \eqref{problem} is nonempty and compact.
\end{theorem}

\section{Applications to vector optimization problems}\label{Section 4}
In this    section, we consider the following vector optimization problem:
\begin{equation}\label{VOP}
	\mathrm{Min}_{\R^m_+}\,\{f(x)\;|\; x\in\mathcal{F}\},\tag{VP}
\end{equation}
where the feasible set $\mathcal{F}$ is defined by \eqref{feasible-set} and $\R^m_+$ is the nonnegative orthant of $\R^m$ and given by $$\R^m_+:=\{x=(x_1, \ldots, x_m)\in\R^m\;|\; x_i\geq 0, i\in I\}.$$

\begin{definition}
	{\rm Let $y\in \cl f(\mathcal{F})$. We say that
		\begin{enumerate}[(i)]
			\item $y$ is a {\em  weak Pareto (optimal) value} of \eqref{VOP} if
			\begin{equation*}
				f(\mathcal{F})\cap(y-\mathrm{int}\,\R^m_+)=\emptyset.
			\end{equation*}
			The set of all weak Pareto values of \eqref{VOP} is denoted by $\mathrm{val}^w\,\eqref{VOP}$.
			\item  $y$ is a {\em  Pareto (optimal) value} of \eqref{VOP} if
			\begin{equation*}
				f(\mathcal{F})\cap(y-\R^m_+\setminus\{0\})=\emptyset.
			\end{equation*}
			The set of all  Pareto values of \eqref{VOP} is denoted by $\mathrm{val}\,\eqref{VOP}$.
			
			\item A point $\bar x\in\mathcal{F}$ is said to be a {\em  weak Pareto solution} (resp., {\em Pareto solution}) of problem \eqref{VOP} if   $f(\bar x)$ is a weak  Pareto value (resp. Pareto value) of \eqref{VOP}. The set of all weak Pareto solutions (resp., Pareto solutions) of \eqref{VOP} is denoted by $\mathrm{sol}^w\,\eqref{VOP}$ (resp., $\mathrm{sol}\,\eqref{VOP}$).
		\end{enumerate}
	}
\end{definition}
\begin{remark}\rm  By definition, it is clear that
	\begin{equation*}
		\mathrm{val}\,\eqref{VOP} \subset \mathrm{val}^w\,\eqref{VOP} \ \ \text{and}\ \ \mathrm{sol}\,\eqref{VOP}\subset \mathrm{sol}^w\,\eqref{VOP}. 
	\end{equation*}
	
\end{remark}

The following theorem presents a necessary optimality condition of KKT type for a weak Pareto value at infinity of problem \eqref{VOP}. 

\begin{theorem}\label{necessary-VOP} Let $\bar y\in\mathrm{val}^w\eqref{VOP}$. Assume that the condition (CQ) holds at infinity.  If there is a sequence $x_k\xrightarrow{\mathcal{F}}\infty$ such that $f(x_k)\to \bar{y}$, then there exist $\alpha:=(\alpha_1, \ldots, \alpha_m)\in\R^m_+\setminus\{0\}$ and $\beta:=(\beta_1, \ldots, \beta_p)\in\R^p_+$ such that  
	\begin{align}\label{equa-8}
		0\in \sum_{i\in I} \alpha_i\partial f_i(\infty)+\sum_{j\in J}\beta_j\partial g_j(\infty)+N_\Omega(\infty).
	\end{align} 
\end{theorem}
\begin{proof} Assume that $\bar y=(\bar y_1, \ldots, \bar y_m)\in \mathrm{val}^w\eqref{VOP}$. For each $i\in I$ let $\hat{f}_i$ be the function defined by 
	$$\hat{f}_i(x)=f_i(x)-\bar y_i,\ \ \ \forall x\in\R^n.$$
	Clearly,  functions $\hat{f}_i$, $i\in I$, are Lipschitz at infinity. We now consider the following minimax programming problem:
	\begin{equation}\label{aux-minimax-problem}
		\min_{x\in\mathcal{F}}\max_{i\in I} \hat{f}_i(x).
	\end{equation}
	For each $x\in\R^n$, put $\widehat{\varphi}(x):=\max_{i\in I} \hat{f}_i(x)$. 
	We will show that $\widehat{\varphi}$ is bounded from below and  zero is the optimal value of problem  \eqref{aux-minimax-problem}. To do this, let us first prove that
	\begin{equation}\label{equa-9}
		\widehat{\varphi}(x)\geq 0\ \ \ \forall x\in\mathcal{F}.
	\end{equation} 
	Indeed, if \eqref{equa-9} does not hold, then there exists $x^\prime\in\mathcal{F}$ such that $\hat{\varphi}(x^\prime)<0$. This means that
	\begin{equation*}
		f_i(x^\prime) -\bar y_i<0, \ \ \forall i\in I,
	\end{equation*}  
	or, equivalently,
	\begin{equation*}
		f(x^\prime)\in (\bar y-\mathrm{int}\,\R^m_+),
	\end{equation*}
	which contradicts the fact that $\bar y$ is a weak Pareto value of \eqref{VOP}. Now since $f(x_k)\to \bar y$, one has  $f_i(x_k)-y_i\to 0$, $i\in I$, as $k\to\infty$. Combining this with the following estimate
	\begin{equation*}
		0\leq \widehat{\varphi} (x_k)\leq \sum_{i\in I}|f_i(x_k)-\bar y_i|
	\end{equation*}
	we obtain $\widehat{\varphi}(x_k)\to 0$ as $k\to\infty$. Hence,   zero is the optimal value of problem  \eqref{aux-minimax-problem}.
	
	By Theorem \ref{KKT-necessary-theorem}, there exist $\alpha:=(\alpha_1, \ldots, \alpha_m)\in\R^m_+\setminus\{0\}$ and $\beta:=(\beta_1, \ldots, \beta_p)\in\R^p_+$ such that  
	\begin{align}\label{equa-10}
		0\in \sum_{i\in I} \alpha_i\partial \hat{f}_i(\infty)+\sum_{j\in J}\beta_j\partial g_j(\infty)+N_\Omega(\infty).
	\end{align} 
	Clearly, condition \eqref{equa-10} implies \eqref{equa-8}  and  we thus complete the proof of the theorem.
\end{proof}

The next result  gives a sufficient condition for the nonemptiness and compactness of the weak  Pareto solution  set of \eqref{VOP}.
\begin{theorem}\label{Thrm-sufficient-1} Assume that $\mathrm{val}^w\eqref{VOP}$ is nonempty. If the condition (CQ) holds at infinity and  
	\begin{equation}\label{equa-13}
		0\notin \mathrm{co}\,\left\{\partial f_1(\infty), \ldots,  \partial f_m(\infty)\right\}+\mathrm{pos}\,\{\partial g_1(\infty), \ldots, \partial g_p(\infty)\}+N_\Omega(\infty),
	\end{equation}
	then $\mathrm{sol}^w\eqref{VOP}$ is nonempty and compact. 
\end{theorem}
\begin{proof} We first claim that $\mathrm{sol}^w\eqref{VOP}$ is nonempty. Indeed, let $\bar y\in\mathrm{val}^w\eqref{VOP}$. Then there exists a sequence $x_k\in{\mathcal{F}}$ such that $f(x_k)\to \bar y$. By assumptions of the theorem and Theorem \ref{necessary-VOP}, the sequence $x_k$ is bounded and so we may assume that $x_k$ converges to some $\bar x\in \mathcal{F}$. By the lower semicontinuity of functions $f_i$ at $\bar x$,  we have
	\begin{equation*}
		f_i(\bar x)\leq \liminf_{x\to\bar x}f_i(x)\leq \lim_{k\to\infty}f_i(x_k)=\bar{y}_i \ \ \forall i\in I.
	\end{equation*} 
	This means that $f(\bar x)\in \bar y-\R^m_+$. We claim that $\bar x\in \mathrm{sol}^w\eqref{VOP}$. This means that $\mathrm{sol}^w\eqref{VOP}$ is nonempty.  Indeed, if otherwise, then there exists $x\in \mathcal{F}$ such that $f(x)\in f(\bar x) -\mathrm{int}\R^m_+$. This implies that
	\begin{align*}
		f(x)\in f(\bar x) -\mathrm{int}\R^m_+ &\subset (\bar y-\R^m_+)-\mathrm{int}\R^m_+
		\\
		&\subset \bar y-\mathrm{int}\R^m_+,
	\end{align*}	
	which contradicts the fact that $\bar y\in\mathrm{val}^w\eqref{VOP}$.  
	
	Now we will prove that  $\mathrm{sol}^w\eqref{VOP}$ is bounded. Suppose on the contrary there exists a sequence $x_k$ in  $\mathrm{sol}^w\eqref{VOP}$ such that $x_k\to\infty$ as $k\to\infty$.  Using similar arguments applied in the proof of Theorem \ref{necessary-VOP}, for each $k$ we have $x_k\in \mathrm{argmin}_{\mathcal{F}}\,\max_{i\in I}(f_i(x)-f_i(x_k))$. By \cite[Corollary 6.6]{Mordukhovich2018}, for $k$ large enough there exist $\alpha^i_k\geq 0$, $u^i_k\in \partial f_i(x_k)$, $i\in I$, $\beta^j_k\geq 0$, $v^j_k\in\partial g_j(x_k)$, $j\in J$, and $w_k\in N_\Omega(x_k)$ such that $\sum_{i\in I}\alpha^i_k+\sum_{j\in J}\beta^j_k=1$ and
	\begin{equation}\label{equa-11}
		\sum_{i\in I}\alpha^i_ku^i_k+\sum_{j\in J}\beta^j_kv^j_k+w_k=0.
	\end{equation}
	By the Lipschitz at infinity property  of $f$ and $g$,  sequences $u^i_k$ and $v^j_k$ are bounded for all $k$ large enough. By passing to a subsequence if necessary we can assume that $\alpha^i_k\to \alpha_i$, $\beta^j_k\to \beta_j$, $u^i_k\to u_i\in\partial f_i(\infty)$, $v^j_k\to v_j\in\partial g_j(\infty)$ such that $\sum_{i\in I}\alpha_i+\sum_{j\in J}\beta_j=1$. By \eqref{equa-11}, the sequence $w_k$ also converges to some $w\in N_{\Omega}(\infty)$.
	Hence,
	\begin{equation}\label{equa-12}
		\sum_{i\in I}\alpha_i u_i+\sum_{j\in J}\beta_j v_j+w=0.
	\end{equation}
	If $\alpha_i=0$ for all $i\in I$, then by \eqref{equa-12}, $0\in \sum_{j\in J}\beta_j\partial g_j(\infty)+N_\Omega(\infty)$, where $\sum_{j\in J}\beta_j=1$, which contradicts to the fact that the condition (CQ)    holds  at infinity. Hence, $\sum_{i\in I}\alpha_i>0$ and so
	\begin{equation*}
		0\in \mathrm{co}\,\left\{\partial f_1(\infty), \ldots,  \partial f_m(\infty)\right\}+\mathrm{pos}\,\{\partial g_1(\infty), \ldots, \partial g_p(\infty)\}+N_\Omega(\infty),
	\end{equation*}
	contrary to \eqref{equa-13}, as required.  
	
	To finish the proof, we need to show that $\mathrm{sol}^w\eqref{VOP}$ is closed and so this set is compact. To proceed, let $x_k$ be a sequence in $\mathrm{sol}^w\eqref{VOP}$ that converges to some $\bar x$. We claim that $\bar x\in \mathrm{sol}^w\eqref{VOP}$ and so $\mathrm{sol}^w\eqref{VOP}$ is closed. Indeed, if otherwise, $\bar x\notin \mathrm{sol}^w\eqref{VOP}$, then by definition there exist $x^\prime\in\mathcal{F}$ such that 
	$$f(x^\prime)\in f(\bar x)-\mathrm{int}\, \R^m_+,$$
	or, equivalently,
	\begin{equation}\label{equa-14}
		f_i(x^\prime)<f_i(\bar x) \ \ \ \forall i\in I.
	\end{equation}
	Now by the lower semicontinuity of $f_i$  at $\bar x$,   we have
	\begin{equation*}
		f_i(\bar x)\leq \liminf_{x\to\bar x}f_i(x)\leq \lim_{k\to\infty}f_i(x_k) \ \ \ \forall i\in I.
	\end{equation*}
	This and \eqref{equa-14} imply that for all $k$ large enough
	\begin{equation*}
		f_i(x^\prime)<f_i(x_k) \ \ \ \forall i\in I.
	\end{equation*}
	Hence, 
	$$f(x^\prime)\in f(x_k)-\mathrm{int}\,\R^m_+ \ \ \text{for all}\ \ k \ \ \text{large enough},$$ 
	which contradicts to the fact that $x_k\in \mathrm{sol}^w\eqref{VOP}$. The proof is complete.
\end{proof}

We finish this section by presenting a result that gives a sufficient condition for the nonemptiness and boundedness of the Pareto solution  set of \eqref{VOP}.
\begin{theorem}\label{Thrm-sufficient-2} Assume that $\mathrm{val}\eqref{VOP}$ is nonempty. If the condition (CQ) holds at infinity and  
	\begin{equation*} 
		0\notin \mathrm{co}\,\left\{\partial f_1(\infty), \ldots,  \partial f_m(\infty)\right\}+\mathrm{pos}\,\{\partial g_1(\infty), \ldots, \partial g_p(\infty)\}+N_\Omega(\infty),
	\end{equation*}
	then $\mathrm{sol}\eqref{VOP}$ is nonempty and bounded. 
\end{theorem}
\begin{proof} Let $\bar y\in \mathrm{val}\eqref{VOP}$. Then, by a similar argument as in the proof of Theorem \ref{Thrm-sufficient-1}, there exists a sequence $x_k\in \mathcal{F}$ such that $f(x_k)\to \bar y$, $x_k$ converges to some $\bar x\in \mathcal{F}$ and $f(\bar x)\in \bar y-\mathbb{R}^m_+$. We show that $\bar x\in \mathrm{sol}\eqref{VOP}$, i.e., $\mathrm{sol}\eqref{VOP}$ is nonempty. Indeed, if otherwise, there is a point $x\in\mathcal{F}$ satisfying $f(x)\in f(\bar x)-\mathbb{R}^m_+\setminus\{0\}$. Hence,
	\begin{align*}
		f(x)\in f(\bar x) -\mathbb{R}^m_+\setminus\{0\}&\subset (\bar y-\R^m_+)-\mathbb{R}^m_+\setminus\{0\}
		\\
		&\subset \bar y-\mathbb{R}^m_+\setminus\{0\},
	\end{align*}	
	which contradicts the fact that $\bar y\in\mathrm{val}\eqref{VOP}$. The bounded of $\mathrm{sol}\eqref{VOP}$ follows directly from Theorem \ref{Thrm-sufficient-1} and the fact that $\mathrm{sol}\eqref{VOP}\subset \mathrm{sol}^w\eqref{VOP}$.	
\end{proof}
\section{Conclusions}
In this paper, necessary and sufficient optimality condition  at infinity for nonsmooth minimax programming problems are investigated. The tools used here are the limiting subdifferential and normal cone   at infinity.  The obtained results are then applied to nonsmooth vector optimization problems. Regarding to Theorem \ref{Thrm-sufficient-2}, the closedness of $\mathrm{sol}\eqref{VOP}$ is still unclear and we leave here as an open question for future research.

\section*{Acknowledgments} A part of this work was done while the first author was visiting Department of Applied Mathematics, Pukyong National University, Busan, Korea in May 2023. The  author would like to thank the department for hospitality and support during their stay.

\section*{Funding} The first author is  supported by Vietnam National Foundation for Science and Technology Development (NAFOSTED), Grant  No. 101.01-2023.23.

\end{document}